\documentclass[10pt]{amsart}  

\usepackage{amsmath,amssymb,amsfonts,graphicx, hyperref,color,pdfpages}

\newtheorem{defin}{Definition}[section]

\newtheorem{theorem}[defin]{Theorem}

\newtheorem{lemma}[defin]{Lemma}

\DeclareMathOperator{\vol}{vol}
\DeclareMathOperator{\GL}{GL}

\title{New dense superball packings in three dimensions}

\author{Maria Dostert} 
\address{M.~Dostert, EPFL, SB TN, Station 8, CH-1015 Lausanne, Switzerland}
\email{maria.dostert@epfl.ch}

\author{Frank Vallentin} 
\address{F.~Vallentin, Mathematisches Institut, Universit\"at zu
  K\"oln, Weyertal~86--90, 50931 K\"oln, Germany}
\email{frank.vallentin@uni-koeln.de}



\date{June 27, 2018}

\begin{document}

\begin{abstract} 
  In this paper we construct a new family of lattice packings for
  superballs in three dimensions (unit balls for the $l^p_3$ norm)
  with $p \in (1, 1.58]$. We conjecture that the family also exists
  for $p \in (1.58, \log_2 3 = 1.5849625\ldots]$. Like in the densest
  lattice packing of regular octahedra, each superball in our family
  of lattice packings has $14$ neighbors.
\end{abstract}

\maketitle

\markboth{M.~Dostert and F.~Vallentin}{New dense superball packings in three dimensions}

\section{Introduction}

Finding dense packings of spheres in $n$-dimensional Euclidean space
is one of the most central problems in discrete geometry. In this
paper, we consider lattice packings of superballs in dimension
three. \textit{Superballs} are unit balls for the $\ell^p_n$ norm:
\[ 
\mathcal{B}^p_n = \{ x \in \mathbb{R}^n : \|x\|_p \leq 1\} \;
\text{where} \; \|x\|_p = \left(\sum_{i = 1}^n |x_i|^p\right)^{1/p}.
\]

We distinguish between three different kinds of packings which are
increasingly restrictive: congruent packings, translative packings,
and lattice packings. Let $\mathcal{K}$ be a convex body in
$\mathbb{R}^n$. By $\mathcal{K}^\circ$ we denote its topological
interior. We consider the \textit{special orthogonal group}
\[
\mathrm{SO}(n) = \{ A \in \mathbb{R}^{n \times n} : AA^{\sf T} = I_n,\; \det A = 1\},
\]
and $I_n$ denotes the identity matrix.

A \textit{congruent packing} of $\mathcal{K}$ is defined as
\[
\mathcal{P} = \bigcup_{i \in \mathbb{N}} \left(x_i + A_i
  \mathcal{K}\right) \text{ with } (x_i, A_i) \in \mathbb{R}^n \times
\mathrm{SO}(n),
\]
where
\[
(x_i + A_i \mathcal{K}^\circ) \cap (x_j + A_j \mathcal{K}^\circ) =
\emptyset \; \text{ whenever } \; i \neq j.
\]
\textit{Translative packings} are congruent packings without
rotations, i.e.\ all matrices $A_i$ are equal to $I_n$. Moreover, if
the set $\{x_i : i \in \mathbb{N}\}$ forms a lattice (a discrete
subgroup of $\mathbb{R}^n$), then we call this packing a
\textit{lattice packing}.

The \textit{(upper) density} of a congruent packing $\mathcal{P}$ is
defined as
\[
\delta(\mathcal{P}) = \limsup_{r \rightarrow \infty} \sup_{c \in
  \mathbb{R}^n} \frac{\vol(B(c,r) \cap \mathcal{P})}{\vol B(c,r)},
\]
where $B(c,r) = \{x \in \mathbb{R}^n : \|x-c\|_2 \leq r\}$ is the
Euclidean ball of radius $r$ with center $c$.

Jiao, Stillinger, and Torquato \cite{Jiao2009a, Jiao2011a} determined
four families of dense superball packings via computer
simulations. They divided the possible values of $p$ into four
regimes:
\[
p \in [1, \infty) = [1, \log_2 3 = 1.5849\ldots] \cup [\log_2 3, 2]
\cup [2, 2.3018\ldots] \cup [2.3018\ldots, \infty).
\]
For every regime they found a family of lattices which defines dense
lattice packings for $\mathcal{B}^p_3$. The family for the first
regime $[1,\log_2 3]$ they call the $\mathbb{O}_1$ lattices, and
$\mathbb{O}_0$, $\mathbb{C}_0$, $\mathbb{C}_1$ correspondingly for the
second, third, and fourth regime. In their approach, they even allowed
for congruent packings, but it turned out that the densest congruent
packings they obtained were lattice packings. They conjectured that
they found the densest congruent packings of $\mathcal{B}^p_3$ for all
values of $p$.

However, Ni, Gantapara, de Graaf, van Roij, and Dijkstra
\cite{Ni2012a} found, also via computer simulations, denser lattice
packings for superballs in the second regime. Also for $p = 1.4$,
which falls into the first regime, they give one lattice for which
they claim (but see Section~\ref{sec:Findings})) that it is denser
than the corresponding $\mathbb{O}_1$ lattice. They report (see page
8830 in \cite{Ni2012a}) that due to numerical instabilities they could
not investigate values $p < 1.4$.

In this paper, we use a method of Minkowski \cite{Minkowski1904a} to
determine locally densest lattice packings of
$\mathcal{B}^p_3$. Minkowski applied this method to determine the densest
lattice packing of regular octahedra $\mathcal{B}^1_3$. Starting from
Minkowski's lattice we found a family of new lattice packings for the
first regime which is denser than the $\mathbb{O}_1$ lattices, see
Table~\ref{table:FirstRegime}.  Here, each superball in our family of
lattice packings has $14$ neighbors, like in the densest lattice
packing of regular octahedra. We also found new lattice packings in
the second regime which are denser than the $\mathbb{O}_0$ lattices,
see Table~\ref{table:SecondRegime}. For the third and fourth regime
the densest lattices we found are the $\mathbb{C}_0$ lattices,
respectively the $\mathbb{C}_1$ lattices.

{\small
\begin{table}[htb]
\begin{center}
\renewcommand{\arraystretch}{1.2}
\begin{tabular}{ccc}
$p$ & \textsl{packing density of $\mathbb{O}_1$ lattices} & \textsl{packing density of new family} \\[0.5ex]
$1$  & $18/19 = 0.94736\ldots$ & $0.94736\ldots$ \\
$1.1$  & $0.90461\ldots$ & $0.90913\ldots$ \\
$1.2$  & $0.87121\ldots$ & $0.87861\ldots$ \\
$1.3$  & $0.84516\ldots$ & $0.85375\ldots$ \\
$1.4$  & $0.82497\ldots$ & $0.83284\ldots$ \\
$1.5$  & $0.80948\ldots$ & $0.81395\ldots$ \\
$\log_2 3$ & $0.79594\ldots$ & $0.79594\ldots$ \\
\end{tabular}
\\[0.3cm]
\end{center}
\caption{Packing density of $\mathbb{O}_1$ lattices (see
  \cite{Jiao2009a, Jiao2011a}) and of our new family (see
  Section~\ref{sec:Findings} and Section~\ref{sec:Family}). When $p = 1$ we obtained the lattice
  which determines the densest lattice packing of regular octahedra
  (see \cite{Minkowski1904a}) and when $p = \log_2 3$ we obtain the
  body centered cubic lattice.}
\label{table:FirstRegime}
\end{table}
}

{\small
\begin{table}[htb]
\begin{center}
\renewcommand{\arraystretch}{1.2}
\begin{tabular}{ccc}
$p$ & \textsl{packing density of $\mathbb{O}_0$ lattices} & \textsl{packing density of new lattices} \\[0.5ex]
$\log_2 3$ & $0.79594\ldots$ & $0.79594\ldots$ \\
$1.6$ & $0.79084\ldots$  & $0.79084\ldots$ \\
$1.7$ & $0.76567\ldots$  & $0.76610\ldots$ \\
$1.8$ & $0.75126\ldots$  & $0.75303\ldots$ \\
$1.9$ & $0.74364\ldots$  & $0.74550\ldots$ \\
$2$ & $\pi/\sqrt{18} =  0.74048\ldots$ & $0.74048\ldots$ \\
\end{tabular}
\\[0.3cm]
\end{center}
\caption{Packing density of $\mathbb{O}_0$ lattices (see
  \cite{Jiao2009a, Jiao2011a}) and of the new lattices (see
  \cite{Ni2012a} and see Section~\ref{sec:Findings}). When $p = 2$ we
  obtain the face centered cubic lattice.}
\label{table:SecondRegime}
\end{table}
}

The structure of the paper is as follows: We first explain Minkowski's
method and how we approached it computationally in
Section~\ref{sec:Determination}. Section~\ref{sec:Findings} contains a
report on some of our findings, in particular new lattice packings in
the first and second regime.  Section~\ref{sec:Family} provides a
computer-assisted proof proving that the family of new dense lattice
packings exists for $p \in [1,1.58]$. We end with
Section~\ref{sec:Conjectures} by posing some conjectures and open
problems.

\section{Determination of lattice packings by local optimization}
\label{sec:Determination}

Let $\Lambda$ be a lattice so that 
\[
\mathcal{P} = \bigcup_{x \in \Lambda} (x + \mathcal{K})
\]
forms a lattice packing of the convex body $\mathcal{K}$, i.e.\ for
all distinct lattice vectors $x, y \in \Lambda$ we have
\begin{equation}
\label{eq:packingcondition}
(x + \mathcal{K}^\circ) \cap (y + \mathcal{K}^\circ) = \emptyset.
\end{equation}
Then, $\Lambda$ is called a \textit{packing lattice}. From
\eqref{eq:packingcondition} it immediately follows that a packing
lattice is characterized by
\[
\Lambda \cap (\mathcal{K} - \mathcal{K})^\circ = \emptyset, \; \text{ where } 
\mathcal{K} - \mathcal{K} = \{x - y : x, y \in \mathcal{K}\}
\]
is the \textit{difference body} of $\mathcal{K}$. In other words,
$\Lambda$ is a packing lattice for $\mathcal{K}$ if and only if the
condition
\[
\|x\|_{\mathcal{K} - \mathcal{K}} \geq 1 \; \text{ for all } x \in \Lambda \setminus \{0\}
\]
holds; here
\[
\|x\|_{\mathcal{L}} = \inf\left\{ \lambda : \frac{1}{\lambda} x \in \mathcal{L}\right\}
\]
is the \textit{Minkowski norm} of $x$ defined by a centrally symmetric
convex body $\mathcal{L}$.

The \textit{general linear group} of degree $n$ over a ring $R$ is defined as
\[
\GL_n(R) = \{ B \in R^{n \times n} : \exists A \in R^{n \times n}: AB = BA = I_n\}.
\]
A matrix $B \in \GL_n(\mathbb{R})$ with linearly independent column
vectors $b_1, \ldots, b_n$ specifies a lattice $\Lambda$ by taking all
integral linear combinations of $b_1, \ldots, b_n$:
\[
\Lambda = \mathbb{Z} b_1 + \cdots + \mathbb{Z} b_n = B\mathbb{Z}^n.
\]
Two matrices $B, B' \in \GL_n(\mathbb{R})$ determine the same lattice
if and only if there is a matrix $U \in \GL_n(\mathbb{Z})$ such that
$BU = B'$. Matrix $B$ also gives a fundamental domain $F$ of $\Lambda$ by
\[
F = \left\{ \sum_{i=1}^n \alpha_i b_i : \alpha_i \in [0,1], i = 1, \ldots, n\right\}
\]
Then, the volume of a fundamental domain of $\Lambda$ is
$\vol(\mathbb{R}^n / \Lambda) = \vol F = |\det(B)|$. The density of a
lattice packing
$\mathcal{P} = \bigcup_{x \in \Lambda} (x + \mathcal{K})$ is
\[
\delta(\mathcal{P}) = \frac{\vol \mathcal{K}}{\vol(\mathbb{R}^n /
  \Lambda)} =\frac{\vol \mathcal{K}}{|\det B|}.
\]
So one can state the problem of finding a densest lattice packing of a
convex body $\mathcal{K}$ as the following minimization problem:
\[
\begin{split}
\text{minimize } & |\det B| \\
\text{so that } & B \in \GL_3(\mathbb{R}) / \GL_3(\mathbb{Z})\\
& \|Bu\|_{\mathcal{K} - \mathcal{K}} \geq 1 \text{ for all } u \in \mathbb{Z}^n\setminus\{0\}.
\end{split}
\]
It follows from Mahler's selection theorem that the minimum is
attained, see \cite[Theorem 30.1]{Gruber2007a}.

Two distinct translates $x + \mathcal{K}$ and $y + \mathcal{K}$ are
called \textit{neighbors} in a lattice packing $\mathcal{P}$ if they
have a nonempty intersection. The number of neighbors coincides for
all translates. How many neighbors can $\mathcal{K}$ have at most?
Minkowski showed that $\mathcal{K}$ has at most $3^n-1$ neighbors and
if $\mathcal{K}$ is strictly convex, then the number of neighbors is
bounded by $2^{n+1}-2$, see \cite[Theorem
30.2]{Gruber2007a}. Swinnerton-Dyer proved that when a lattice $\Lambda$
achieves a density which is locally maximal, then $\mathcal{K}$ has at
least $n(n+1)$ neighbors, see \cite[Theorem 30.3]{Gruber2007a}.

Let $B \in \GL_3(\mathbb{R})$ be a matrix defining a locally densest
lattice packing $\Lambda = B\mathbb{Z}^3$ of $\mathcal{K}$. Minkowski
\cite{Minkowski1904a} (see also \cite[\S 32]{Gruber1987a},
\cite{Betke2000a}, \cite{Haus1999a}) showed that after performing a
suitable $\GL_n(\mathbb{Z})$-transformation to $B$ we can reduce to
the following three cases.

\noindent \textbf{Case (I):} $Bu + \mathcal{K}$ with $u \in \mathcal{U}^1$ are neighbors of $\mathcal{K}$, where
\[
\mathcal{U}^1 = \{\pm (1,0,0), \pm (0,1,0), \pm (0,0,1), \pm (1,-1,0), \pm (0, 1,-1), \pm (1,0,-1)\},
\]
but $\pm (-1,1,1), \pm (1,-1,1), \pm (1,1,-1)$ are not.

\noindent \textbf{Case (II):} $Bu + \mathcal{K}$, with $u \in \mathcal{U}^2$, are neighbors of $\mathcal{K}$, where
\[
\mathcal{U}^2 = \{\pm (1,0,0), \pm (0,1,0), \pm (0,0,1), \pm (1,1,0), \pm (0, 1,1), \pm (1,0,1)\},
\]
but $\pm (1,1,1)$ are not.

\noindent \textbf{Case (III):} $Bu + \mathcal{K}$, with $u \in \mathcal{U}^3$, are neighbors of $\mathcal{K}$, where
\[
\mathcal{U}^3 = \mathcal{U}^2 \cup \{\pm (1,1,1)\}.
\]

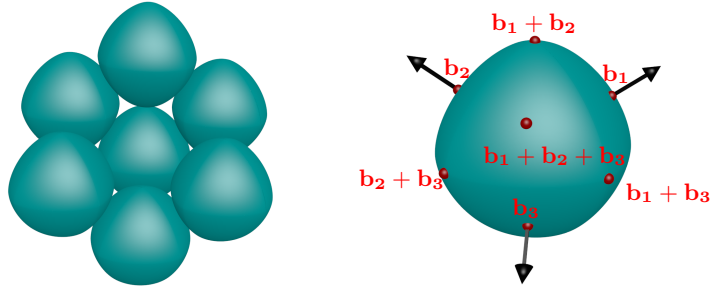
\begin{figure}[htb]
  \centering
    \includegraphics[scale=0.12]{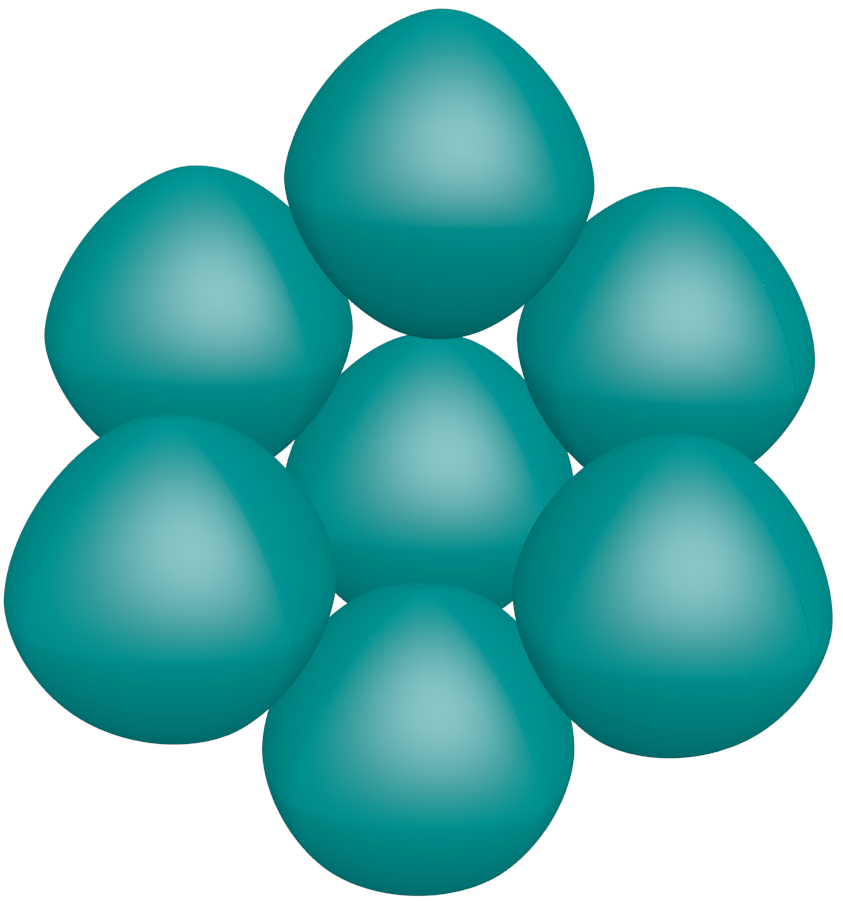}  
    \hspace*{15mm}
    \includegraphics[scale=0.13]{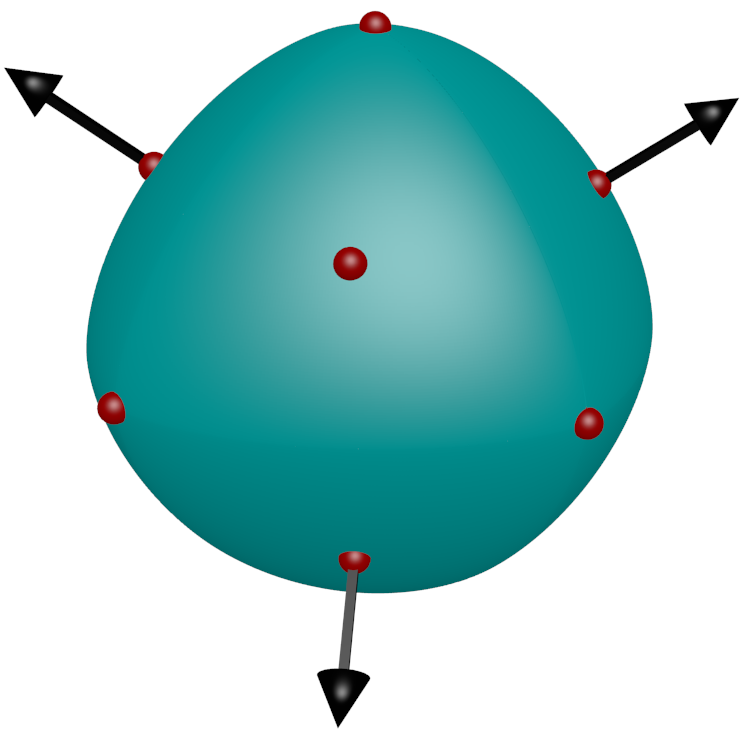}
  \put(-57,27){\small \textcolor{red}{$\bf{b_3}$}}
  \put(-115,38){\small \textcolor{red}{$\bf{b_2+b_3}$}}
   \put(-23,78){\small \textcolor{red}{$\bf{b_1}$}}
   \put(-14,33){\small \textcolor{red}{$\bf{b_1 + b_3}$}}
  \put(-68,47){\small \textcolor{red}{$\bf{b_1 + b_2 + b_3}$}}
  \put(-65,97){\small \textcolor{red}{$\bf{b_1 + b_2}$}}
  \put(-83,80){\small \textcolor{red}{$\bf{b_2}$}}
  \caption{On the left: A part of a lattice packing for $B^{\log_2 3}_3$ satisfying Case (III). On the right: Seven of fourteen neighbors of the packing. Contact points are labeled in red. }
  \label{fig:Neighbors}
\end{figure}

From now on, we are interested in lattice packings of superballs in three
dimensions. We perform a rescaling by setting
$\mathcal{K} = \frac{1}{2}\mathcal{B}^p_3$ so that
$\|x\|_{\frac{1}{2}\mathcal{B}^p_3 - \frac{1}{2} \mathcal{B}^p_3} =
\|x\|_p$.
Note that scaling $\mathcal{K}$ only scales the packing lattices but
does not effect the packing density.

For each of the three cases, Case (I), (II), and (III), we can
numerically find critical points of the following finite nonlinear
optimization problem
\begin{equation}
\label{eq:nonlinear}
\begin{split}
\text{minimize } & \det B \\
\text{so that } & B \in \GL_3(\mathbb{R})\\
& \det B > 0\\
& \|Bu\|_p = 1 \text{ for all } u \in \mathcal{U}^i,
\end{split}
\end{equation}
for $i = 1, 2, 3$, in order to identify candidates for locally densest
lattice packings.

After finding a matrix $B$ which is a feasible solution of this
optimization problem, we have to check whether $B$ indeed defines a
packing lattice for $\mathcal{B}^p_3$. For this it suffices to verify
that $\|Bu\|_p \geq 1$ holds for only finitely many vectors
$u \in \mathbb{Z}^n \setminus \{0\}$ as the following lemma shows; see
also Dieter~\cite{Dieter1975a}.

\begin{lemma}
  Suppose $p,q \in [1,\infty)$ satisfy the relation
  $\frac{1}{p} + \frac{1}{q} = 1$ and suppose that a matrix
  $B \in \GL_n(\mathbb{R})$ is given. If
  $u = (u_1, \ldots, u_n) \in \mathbb{Z}^n$ is such that
  $\|Bu\|_p \leq \mu$ for some nonnegative number $\mu$, then
\[
|u_i| \leq \|B^{-1}_i\|_q \; \mu \quad \text{ for all } i = 1, \ldots, n,
\]
where $B^{-1}_i$ is the $i$-th row vector of $B^{-1}$.
\end{lemma}

\begin{proof}
We apply the triangle inequality and H\"older's inequality, and get
\[
|u_i| = |B^{-1}_i Bu| \leq \sum_{j=1}^n |B^{-1}_{i,j} (Bu)_j|
\leq \|B^{-1}_i\|_q \|Bu\|_p \leq \|B^{-1}_i\|_q \; \mu.
\qedhere
\]
\end{proof}

If $1 < p < 2$ and if $B \in \GL_3(\mathbb{R})$ satisfies the equality
conditions of Case~(III)
\[
\|Bu\|_p = 1 \text{ for all } u \in \mathcal{U}^3,
\]
then this extra check is not necessary, as the next lemma shows. This is
a consequence of Hanner's inequality \cite{Hanner1956a} 
\begin{equation}
\label{eq:Hanner}
\|x + y\|^p_p + \|x - y\|^p_p \geq (\|x\|_p + \|y\|_p)^p + \left| \|x\|_p - \|y\|_p \right|^p,
\end{equation}
which holds for all $1 < p < 2$ and all $x, y \in \mathbb{R}^n$.

\begin{lemma} 
\label{lem:Hanner}
If $1 < p < 2$ and if $B \in \GL_3(\mathbb{R})$ satisfies
$\|Bu\|_p = 1$ for all $u \in \mathcal{U}^3$, then $\|Bu\|_p \geq 1$
for all $u \in \mathbb{Z}^n \setminus \{0\}$.
\end{lemma}

\begin{proof}
  As a start we consider $u = (1,-1,0)$. We write $u$ as the sum of
  two vectors $v, w$ in $\mathcal{U}^3$: $u = (1,0,0) +
  (0,-1,0)$. Then by Hanner's inequality~\eqref{eq:Hanner} we get
\[
\begin{split}
\|Bu\|^p_p & \; = \; \|Bv + Bw\|^p_p \\
& \; \geq \; (\|Bv\|_p + \|Bw\|_p)^p + \left| \|Bv\|_p - \|Bw\|_p \right|^p -  \|Bv - Bw\|^p_p \\
& \; \geq \; 2^p - 1 \; \geq \; 1.
\end{split}
\]
The vectors $u = (1,0,-1), (0,1,-1)$ can be treated similarly.

Now we consider $u = (-1,1,1)$. We write $u$ as $u = v + w$ with
$u = (0,1,1)$, $w = (-1,0,0)$ and apply Hanner's inequality to get
$\|Bu|^p_p \geq 2^p - 1$. Again, the vectors $u = (1,-1,1), (1,1,-1)$
can be treated similarly.

So we have $\|Bu\|_p \geq 1$ for all
$u \in \mathbb{Z}^3 \setminus \{0\}$ with $|u_i| \leq 1$ and we
proceed by induction. For the inductive step assume that
$\|Bu\|_p \geq 1$ for all $u \in \mathbb{Z}^3 \setminus \{0\}$ with
$|u_i| \leq M$. Let $u \in \mathbb{Z}^3$ with $|u_i| \leq M+1$ for all
$i$ and $|u_i| = M+1$ for some
$i$. Define $v, w \in \mathbb{Z}^3$ componentwise by
\[
v_i = \left\lceil \frac{u_i}{2} \right\rceil \; \text{ and } \; w_i =
\left\lfloor \frac{u_i}{2} \right\rfloor .
\]
Then, $|v_i|, |w_i| \leq M$. Hence, by the induction hypothesis,
$Bv\|_p \geq 1$, and $\quad \|Bw\|_p \geq 1$. Since
$v_i - w_i \in \{0,1\}$, we have $\|Bv-Bw\|_p = 1$. Hanner's
inequality~\eqref{eq:Hanner} implies the desired inequality
$\|Bu\|_p^p \geq 2^p - 1 \geq 1$.
\end{proof}

\section{Numerical findings}
\label{sec:Findings}

In order to explore dense lattice packings of superballs in three
dimensions we numerically found critical points of the nonlinear
minimization problem~\eqref{eq:nonlinear}. We randomly chose a matrix
$B \in \GL_3(\mathbb{R})$, where we chose the matrix entries randomly
according to the normal distribution $N(0,1)$ with mean~$0$ and
variance~$1$. Then we applied Newton's method to $B$ to find a
critical point in the neighborhood of $B$. By applying this procedure
to $10,000$ randomly chosen starting points, we obtained a set of
packing lattices. For the implementation we used the function
\verb|root| of the python package \verb|scipy.optimize|, see
\cite{SciPy2018a}. We determined feasible packing lattices for each of
these three cases.

\subsection{First regime}

The highest packing densities over all obtained solutions we found for
$p = 1, 1.1, 1.2, 1.3, 1.4, 1.5$ belong to Case (III). We list them in
Table~\ref{table:CaseIII} and we analyze them in detail in
Section~\ref{sec:Family}. The density we obtained for $p = 1.4$
coincides with the value reported in \cite{Ni2012a}, but the
basis $e_1, e_2, e_3$ given in~\cite[page 8829]{Ni2012a} does not
give a packing lattice since the $\ell^p_3$-norm of 
$2e_1-e_2-e_3$ is too small.

{\small
\begin{table}[htb]
\[
L_{1} = 
\begin{pmatrix}
-0.333333333333 & 0.166666666667 & 0.5\\
0.5 & -0.333333333333 & 0.166666666667\\
0.166666666667 & 0.5 & -0.333333333333\\
\end{pmatrix}
\]
\[
L_{1.1} = 
\begin{pmatrix}
-0.364125450067 & 0.193419513868 & 0.539049770666\\
0.539049770666 & -0.364125450067 & 0.193419513867\\
0.193419513867 & 0.539049770666 & -0.364125450068\\
\end{pmatrix}
\]
\[
L_{1.2} = 
\begin{pmatrix}
-0.392613644302 & 0.22381214158 & 0.569113821114\\
0.569113821115 & -0.392613644298 & 0.223812141583\\
0.223812141575 & 0.569113821114 & -0.392613644306\\
\end{pmatrix}
\]
\[
L_{1.3} = 
\begin{pmatrix}
-0.419839537546 & 0.260336714788 & 0.589023079183\\
0.589023079194 & -0.419839537534 & 0.260336714788\\
0.260336714754 & 0.589023079183 & -0.419839537578\\
\end{pmatrix}
\]
\[
L_{1.4} = 
\begin{pmatrix}
-0.446984776893 & 0.307534456657 & 0.595696355817\\
0.595696355844 & -0.446984776872 & 0.307534456649\\
0.307534456588 & 0.595696355814 & -0.446984776962\\
\end{pmatrix}
\]
\[
L_{1.5} = 
\begin{pmatrix}
-0.475292821919 & 0.375983627555 & 0.580059051165\\
0.580059051205 & -0.475292821888 & 0.375983627545\\
0.375983627482 & 0.58005905116 & -0.475292821997\\
\end{pmatrix}
\]
\caption{Matrices $L_p \in \GL_3(\mathbb{R})$ giving the densest
  lattice packing of $\mathcal{B}^p_n$ we found. They all belong to
  Case (III). }
\label{table:CaseIII}
\end{table}
}

\subsection{Second regime}

The highest packing density over all obtained solutions we found for
$p = 1.6, 1.7, 1.8, 1.9, 2.0$ belong to Case (I). We list them in
Table~\ref{table:CaseI}. We were not able to identify a pattern. The
density we obained for $p = 1.7$ coincides with the value reported in
\cite{Ni2012a}, but in \cite{Ni2012a} a corresponding
matrix is not given.  For more computational results we refer to the
thesis~\cite{Dostert2017a} of the first author.

{\small
\begin{table}[htb]
\[
L_{1.6} = 
\begin{pmatrix}
-0.000274732684343 & 0.00144215026174 & -0.999980951403\\
0.509783945989 & 0.500572697171 & -0.499408517681\\
0.509509213365 & -0.499408255119 & -0.500850623006\\
\end{pmatrix}
\]
\[
L_{1.7} = 
\begin{pmatrix}
0.458033772615 & 0.556735224273 & 0.553047497039\\
-0.530691863753 & 0.577007354869 & 0.459643833129\\
-0.0936769789857 & -0.0202829019484 & 0.988672468644\\
\end{pmatrix}
\]
\[
L_{1.8} = 
\begin{pmatrix}
-0.330208442415 & -0.696395141028 & 0.551458413193\\
0.624661955256 & -0.637870063365 & 0.316559795406\\
0.237379400053 & -0.0588146621535 & 0.954027742247\\
\end{pmatrix}
\]
\[
L_{1.9} =
\begin{pmatrix}
-0.325366212309 & -0.0828873750566 & 0.930867632285\\
0.230698700286 & 0.676231149106 & 0.66666839749\\
-0.697856599406 & 0.59207566084 & 0.335768664213\\
\end{pmatrix}
\]
\[
L_2 =
\begin{pmatrix}
0.000000000000 & 0.707106781187 & 0.707106781187\\
0.707106781623 & 0.00000000000 & 0.70710678075\\
0.707106781623 & 0.70710678075 & 0.000000000000 \\
\end{pmatrix}
\]
\caption{Matrices $B_p \in \GL_3(\mathbb{R})$ giving the densest
  lattice packing of $\mathcal{B}^p_n$ we found. They all belong to
  Case (I). $L_2$ determines the face centered cubic lattice which
  defines a densest sphere packing in three dimensions.}
\label{table:CaseI}
\end{table}
}

\section{A new family of lattice packings}
\label{sec:Family}

When looking at the numerical solutions we found in the first regime,
one immediately comes to the conclusion that the new found lattices
belong to a family of lattices which one can easily parametrize.

Consider
\[
L(x,y,z) \in \GL_3(\mathbb{R}) \quad \text{with} \quad L(x,y,z) =
\begin{pmatrix}
-x & y & z\\
z & -x & y\\
y & z & -x
\end{pmatrix},
\]
where $x, y, z$ are chosen so that
\[
\|L(x,y,z)u\|_p = 1 \text{ for all } u \in \mathcal{U}^3.
\]
This implies 
\[
\begin{split}
& \|L(x,y,z)(1,0,0)\|^p_p = |-x|^p + |y|^p + |z|^p = 1\\
& \|L(x,y,z)(1,1,0)\|^p_p = |-x+y|^p + |z-x|^p + |y + z|^p = 1\\
& \|L(x,y,z)(1,1,1)\|^p_p = 3|-x+y+z|^p = 1.
\end{split}
\]
We also have the inequalities
\[
z \geq x \geq y \geq 0
\]
Thus, $x,y,z$ has to satisfy the following nonlinear system:
\begin{equation}
\label{eq:nonlinearsystem}
\begin{split}
& z \geq x \geq y \geq 0\\
& x^p + y^p + z^p = 1 \\
& (x-y)^p + (z-x)^p + (y+z)^p = 1\\
& 3(-x +y +z)^p = 1\\
\end{split}
\end{equation}

The family of lattices starts at $p = 1$ with
$x = 1/3, y = 1/6, z = 1/2$, which defines the densest lattice packing
of regular octahedra found by Minkowski. The family ends at
$p = \log_2 3 = 1.5849625\ldots$ with $x = y = z = \frac{1}{2}$ which
defines the body centered cubic lattice.

We want to prove that the family indeed exists. When the nonlinear
system~\eqref{eq:nonlinearsystem} has a solution,
Lemma~\ref{lem:Hanner} ensures that the members of the family are
packing lattices for the corresponding superball
$\frac{1}{2}B^p_3$. We will apply the following theorem of Cohn,
Kumar, and Minton \cite[Theorem 3.1]{Cohn2016a}, which is an effective
implicit function theorem.

\begin{theorem} 
\label{thm:ExistenceTheorem}
Let $V$ and $W$ be finite-dimensional normed vector spaces over
$\mathbb{R}$, and suppose that $f: B(x_0, \epsilon) \rightarrow W$ is
a $C^1$ function, where $x_0 \in V$ and $\varepsilon > 0$. Suppose
also that $T:W \rightarrow V$ is a linear operator such that
\begin{align}
\label{eq:ExistenceTheorem}
\|Df(x) \circ T - id_W\| < 1 - \frac{\|T\| \cdot |f(x_0)|}{\varepsilon}
\end{align}
for all $x \in B(x_0, \varepsilon)$. Then there exists
$x_* \in B(x_0, \varepsilon)$ such that $f(x_*) = 0$. Moreover, in
$B(x_0, \varepsilon)$, the zero locus $f^{-1}(0)$ is a $C^1$
submanifold of dimension $\dim V - \dim W$.
\end{theorem}

In the theorem, $\|\cdot\|$ denotes the operator norm, $Df(x)$ is the
Jacobian of $f$ at $x$, and $id_W$ is the identity operator on $W$,
and $B(x_0, \varepsilon) \subseteq V$ is the \textsl{open} ball with
center $x_0$ and radius $\varepsilon$, where the distance is measured
using the norm of $V$.

\begin{theorem}
\label{thm:Existence}
The nonlinear system~\eqref{eq:nonlinearsystem} has a unique solution
$(x_*(p), y_*(p), z_*(p))$ for all $p \in [1,1.58]$. In particular,
the matrix $L(x_*(p),y_*(p),z_*(p))$ defines a packing lattice for
$\frac{1}{2} \mathcal{B}^p_3$.
\end{theorem}

\begin{proof}
Our proof is computer assisted. We use interval arithmetic as implemented in 
the free open source mathematics software system SageMath~\cite{SageMath}.

Define the function $f_p : B((x_0, y_0, z_0), \varepsilon) \to \mathbb{R}^3$ by
\[
f_p(x,y,z) =
\begin{pmatrix}
x^p + y^p + z^p - 1\\
(x-y)^p + (z-x)^p + (y+z)^p - 1\\
3(-x + y + z)^p - 1
\end{pmatrix}.
\]
The Jacobian $Df_p(x,y,z)$ of $f_p$ equals
\[
p
\begin{pmatrix}
x^{p-1} & y^{p-1} & z^{p-1}\\
(x-y)^{p-1}-(z-x)^{p-1} & -(x-y)^{p-1}+(y+z)^{p-1} & (z-x)^{p-1}+(y+z)^{p-1}\\
-3(-x+y+z)^{p-1} & 3(-x+y+z)^{p-1} & 3(-x+y+z)^{p-1}
\end{pmatrix}.
\]
As norms we choose the $\ell^\infty_3$ norm for the domain, as well as
for the codomain. Then the $\ell^\infty_3$ operator norm is the maximum
of the $\ell^1_3$ norms of the rows of the considered matrix.

Now we subdivide the interval $[1.1.58]$ in smaller subintervals and
for every subinterval we choose an $\varepsilon > 0$ and a starting
point $(x_0,y_0,z_0)$ which we found by the numerical solution of the
nonlinear system~\eqref{eq:nonlinearsystem}. 

Our choice of $\varepsilon > 0$ and using the $\ell^{\infty}_3$ norm
ensures that the ball $B((x_0,y_0,z_0), \varepsilon)$ lies in the
region $z \geq x \geq y \geq 0$.  We set $T = Df_p(x_0,y_0,z_0)^{-1}$
and verify that inequality~\eqref{eq:ExistenceTheorem} is satisfied
for all $p$ in the subinterval and for all
$(x,y,z) \in B((x_0,y_0,z_0), \varepsilon)$. Then, the assumptions of
Theorem~\ref{thm:ExistenceTheorem} are satisfied and we can conclude
that the nonlinear system~\eqref{eq:nonlinearsystem} has a unique
solution $(x_*(p), y_*(p), z_*(p))$ for all $p$ in the
subinterval. The verification of~\eqref{eq:ExistenceTheorem} uses
interval arithmetic. Our Sage function \texttt{verify} is only a few
lines long and can be found in Appendix~\ref{sec:a} of this paper.

For example, we choose the subinterval $[1,1.01]$, the starting point
$(x_0,y_0,z_0) = (\frac{1}{3}, \frac{1}{6}, \frac{1}{2})$, and
$\varepsilon = 0.03$. For $T$ we choose
\[
T = 
\begin{pmatrix}
\frac{1}{2} & 0 & -\frac{1}{6}\\
\frac{1}{2} & -\frac{1}{2} & \frac{1}{6}\\
 0 & \frac{1}{2} & 0\\
\end{pmatrix}.
\]
Then, Sage computes
\[
\|Df_p(x,y,z)  T - I_3\| \in [0.00000000000000000,0.035585437892437462]
\]
and
\[
1 - \frac{\|T\| \cdot |f_p(x_0,y_0,z_0)|}{\varepsilon} \in [0.60895579575438163,1.0000000000000000]
\]
for all $(x,y,z) \in B((x_0,y_0,z_0), \varepsilon)$ and for all
$p \in [1,1.01]$. This shows that~\eqref{eq:nonlinearsystem} has a
unique solution for each $p \in [1,1.01]$.

More examples of our choises can be found in Appendix~\ref{sec:b}; all
choices can be found as ancillary file from the \texttt{arXiv.org} e-print archive.
\end{proof}

We conjecture that the family of lattices also exists for
$p \in (1.58, \log_2 3 = 1.5849625007\ldots]$. We could enlarge the interval for which
Theorem~\ref{thm:Existence} holds by increasing the precision and by
using smaller subintervals. For example, running the following Sage
code shows that the lattice exists for
$p \in [1.5849625, 1.5849625+10^{-10}]$.
\begin{verbatim}
verify(1.5849625, 0.499999999842, 0.499999124646, 0.500000875038,
0.0000002, 0.0000000001)
\end{verbatim}
However, when $p$ approaches $\log_2 3$, then the Jacobian approaches
\[
Df_{\log_2 3}\left(\frac{1}{2},\frac{1}{2},\frac{1}{2}\right) =
\log_3 2
\begin{pmatrix}
\frac{2}{3} & \frac{2}{3} & \frac{2}{3}\\
0 & 1 & 1\\
-2 & 2 & 2\\
\end{pmatrix}
\]
which is singular. Furthermore, there is no $\varepsilon > 0$ so that
the ball $B((\frac{1}{2},\frac{1}{2},\frac{1}{2}), \varepsilon)$ is
contained in the region $z \geq x \geq y$. Currently, we do not know
how to modify the approach to be able to handle these two
difficulties.

\section{Conjectures and open problems}
\label{sec:Conjectures}

Based on Section~\ref{sec:Findings} we pose the
following conjectures and open problems:

\begin{enumerate}
\item The family of lattices determined by~\eqref{eq:nonlinearsystem}
  exists for all $p \in (1.58, \log_2 3]$.
\item The lattices we found in the first regime, and the ones found by
  Jiao, Stillinger, and Torquato \cite{Jiao2009a} in the third and
  fourth regime give the densest lattice packings of superballs for
  the corresponding $p$.
\item It would be interesting to develop a better understanding of the
  densest known lattices in the second regime.
\item Is there a value of $p \neq \infty$ for which the upper bound
  for translative packings of superballs determined
  in~\cite{Dostert2017b} matches the corresponding lower bound?
\item For $p > \log_2 3$ there are no lattices which fall into Case
  (II) or into Case (III).
\item What is the largest value of $p \leq 2$ so that the kissing
  number of $\mathcal{B}^p_3$ superballs is strictly larger than $12$?
\end{enumerate}

\section*{Acknowledgements}

We thank Crist\'obal Guzm\'an and Philippe Moustrou for helpful
discussions.

This material is based upon work supported by the
National Science Foundation under Grant No. DMS-1439786 while the
first author was in residence at the Institute for Computational and
Experimental Research in Mathematics in Providence, RI, during the
"Point Configurations in Geometry, Physics and Computer Science"
semester program. The second author was partially supported by the
SFB/TRR 191 ``Symplectic Structures in Geometry, Algebra and
Dynamics'', funded by the DFG. This project has received funding from
the European Union's Horizon 2020 research and innovation programme
under the Marie Sk\l{}odowska-Curie agreement number~764759.

\begin{appendix}

\section{Source code for proof of Theorem~\ref{thm:Existence}}
\label{sec:a}

We used the following program written in Sage in the proof of Theorem~\ref{thm:Existence}.

{\small
\begin{verbatim}
def f(p,x,y,z):
  return vector([x^p+y^p+z^p-1, (x-y)^p+(z-x)^p+(y+z)^p-1, 3*(-x+y+z)^p-1])

def Df(p,x,y,z):
  pm = p-1
  return p*Matrix([[x^pm, y^pm, z^pm],
                   [(x-y)^pm-(z-x)^pm, -(x-y)^pm+(y+z)^pm, (z-x)^pm+(y+z)^pm],
                   [-3*(-x+y+z)^pm, 3*(-x+y+z)^pm, 3*(-x+y+z)^pm]])

def linfinitynorm(A):
  return max([A.row(0).norm(1),A.row(1).norm(1),A.row(2).norm(1)])

def verify(p0,x0,y0,z0,eps,peps):
  p = RIF(p0,p0+peps)
  x = RIF(x0-eps,x0+eps)
  y = RIF(y0-eps,y0+eps)
  z = RIF(z0-eps,z0+eps)
  T = Df(p0,x0,y0,z0)^(-1)
  A = Df(p,x,y,z)*T - identity_matrix(3)
  lhs = linfinitynorm(A)
  rhs = 1 - linfinitynorm(T)*f(p,x0,y0,z0).norm(infinity)/eps
  print(lhs.str(style='brackets')+'<'+rhs.str(style='brackets')+':'+str(lhs < rhs))
\end{verbatim}
}

\section{Examples of choices made in the proof of Theorem~\ref{thm:Existence}}

\label{sec:b}

{\small
\begin{verbatim}
verify(1.0, 0.333333333333, 0.166666666667, 0.5, 0.03, 0.01)
verify(1.01, 0.336543320255, 0.169227330456, 0.504294897412, 0.03, 0.01)
verify(1.02, 0.339721855623, 0.171809715243, 0.508503843298, 0.03, 0.01)
\end{verbatim}

\smallskip

\begin{center}
$\vdots$
\end{center}

\smallskip

\begin{verbatim}
verify(1.5, 0.475292821919, 0.375983627555, 0.580059051165, 0.03, 0.01)
verify(1.51, 0.47822053429, 0.384961182567, 0.576346694842, 0.03, 0.01)
verify(1.52, 0.481163698665, 0.394556223383, 0.572012690078, 0.006, 0.001)
verify(1.521, 0.48145875646, 0.395553814361, 0.571540873724, 0.006, 0.001)
verify(1.522, 0.481753934423, 0.396558835694, 0.571061553436, 0.006, 0.001)
verify(1.523, 0.482049228267, 0.39757142775, 0.570574584849, 0.006, 0.001)
\end{verbatim}

\smallskip

\begin{center}
$\vdots$
\end{center}

\smallskip

\begin{verbatim}
verify(1.577, 0.497880292399, 0.472696125604, 0.523437325276, 0.006, 0.001)
verify(1.578, 0.498157887988, 0.475000219764, 0.521630841401, 0.006, 0.001)
verify(1.579, 0.498433446144, 0.477421354522, 0.519705097786, 0.006, 0.001)
\end{verbatim}
}

\end{appendix}

\end{document}